\documentclass[a4paper,reqno,11pt]{amsart}

\newfont{\cyr}{wncyr10 scaled 1100}

\usepackage[left=2.7cm,right=2.7cm,top=3.5cm,bottom=3cm]{geometry}

\usepackage{amsthm,amssymb,amsmath,amsfonts,mathrsfs,amscd}
\usepackage{mathtools}
\usepackage{extarrows}
\usepackage[latin1]{inputenc}
\usepackage[all]{xy}
\usepackage{latexsym}
\usepackage{longtable}
\usepackage{xcolor}
\usepackage{comment}
\usepackage[shortlabels]{enumitem}
\usepackage{setspace}
\setdisplayskipstretch{}
\usepackage{multirow}
\usepackage{xcolor,colortbl}
\usepackage{changepage}
\usepackage{float}
\usepackage{tikz-cd}
\usepackage[colorlinks=true,
            linkcolor=black,
            citecolor=black,
            urlcolor=black]{hyperref}
\usepackage[backend=biber,style=alphabetic,maxnames=99,minnames=99, sorting=nyt]{biblatex}

\DeclareFieldFormat[article]{title}{#1}
\DefineBibliographyStrings{english}{
  in = {}
}
\AtBeginBibliography{\footnotesize}

\theoremstyle{plain}
\newtheorem{theorem}{Theorem}[section]
\newtheorem{corollary}[theorem]{Corollary}
\newtheorem{lemma}[theorem]{Lemma}
\newtheorem{proposition}[theorem]{Proposition}

\theoremstyle{definition}
\newtheorem{definition}[theorem]{Definition}

\newtheorem{examplewr}[theorem]{Example}

\theoremstyle{remark}
\newtheorem{obswr}[theorem]{Observation}
\newtheorem{remarkwr}[theorem]{Remark}

\definecolor{Gray}{gray}{0.85}
\definecolor{LightCyan}{rgb}{0.88,1,1}

\newcolumntype{g}{>{\columncolor{Gray}}c}
\newcolumntype{y}{>{\columncolor{LightCyan}}c}
\newcolumntype{o}{>{\columncolor{pink}}c}

\newenvironment{remark}{\begin{remarkwr}\begin{upshape}}{\end{upshape}\end{remarkwr}}

\newcommand{\Z}{\mathbb{Z}}

\newcommand{\A}{\mathbb{A}}

\newcommand{\PP}{\mathbb{P}}

\newcommand{\fp}{\mathfrak{p}}

\makeatletter
\newsavebox{\@brx}
\newcommand{\llangle}[1][]{\savebox{\@brx}{\(\m@th{#1\langle}\)}%
  \mathopen{\copy\@brx\kern-0.5\wd\@brx\usebox{\@brx}}}
\newcommand{\rrangle}[1][]{\savebox{\@brx}{\(\m@th{#1\rangle}\)}%
  \mathclose{\copy\@brx\kern-0.5\wd\@brx\usebox{\@brx}}}
\makeatother

\begin{document}

\include{thebibliography}

\title[A note on cubical Bloch--Levine cycle complexes]{A note on cubical Bloch--Levine cycle complexes}
\author{Peter Xu}

\begin{abstract}
We check that Levine's simplicial--cubical comparison argument for Bloch's cycle complexes also works over an arbitrary DVR. As a result, the sheaf of cubical Bloch cycle complexes computes motivic cohomology for smooth schemes over Dedekind bases.
\end{abstract}

\maketitle
\tableofcontents

\section{Introduction}

In this short note, we show (Theorem \ref{thm:main}) that the cubical and simplicial versions of the Bloch--Levine cycle complexes are quasi-isomorphic over a DVR (or more generally, a semi-local PID where all residue fields are either finite or all infinite). This was shown by Levine over fields \cite[\S4]{L1}, and the arguments are largely the same - the only technical changes amount to some bookkeeping of fibers and residue fields, and a few applications of the Chinese remainder theorem. 

These quasi-isomorphisms come from a zig-zag which which sheafifies over a general base, and so show that the \emph{sheaf} of cubical Bloch--Levine cycles computes the same cohomology groups as the simplicial version treated by Levine in \cite{L2} (which is known to coincide with modern definitions of motivic cohomology for smooth schemes, and has many good properties).

Generally speaking, the main historical interest in cubical cycle complexes was for the natural cubical product, lacking in the simplicial setting. Over a non-field, this product cannot be defined in a homotopy coherent way (in contrast to modern approaches to motivic cohomology) because of an issue with rigidity of cycles in special fibers causing improper intersections, which is perhaps why to date the material of this present article has not been written down. 

However, we have found that in some concrete applications (e.g. our article \cite{X} constructing group cocycles valued in motivic multiple toric polylogarithm classes), it is useful to be able to write cubical cycles to construct motivic cohomology classes and relations between them over general bases: most or all of these could likely be translated into simplicial language by naive simplicial subdivision of cubes, but this results in a less natural formulation and would require additional checking of proper intersections with the newly-introduced simplicial faces. 

In loc. cit., we freely used sheaves of cubical Bloch complexes over a Dedekind base while remarking Levine's original proof should generalize in a straightforward way to non-fields, but decided to write out this note to remove any lingering doubts. We hope others may also find this modest technical exposition helpful if ever encountering similar situations.

\subsection*{Acknowledgements}

Thanks to Thomas Geisser for replying to questions, Tess Bouis for encouraging me to write this short note, and Romyar Sharifi for pointing out the gap in the literature.

\section{Recall of two Bloch--Levine constructions}

Let $R$ be a semi-local PID with fraction field $K$, maximal ideals $\fp_1,\dots,\fp_r$ and corresponding residue fields $k_1,\ldots, k_r$. For simplicity, we will assume either that all residue fields are infinite, or all finite.\footnote{ This assumption is almost surely superfluous, but it makes the descent argument of Section \ref{section:descent} simpler, and we are unaware of any interesting use cases with a mix of finite and infinite residue fields.} Let $X$ be a smooth $R$-scheme of pure relative dimension $d$. Let $\Box^n=(\A^1)^n$ be the algebraic $n$-cube, with codimension-$1$ faces given by the closed subschemes $\{t_i=0\}$ and $\{t_i=1\}$ for various $i$, and correspondingly let 
\[
\Delta^n=\mathrm{Spec}\,R[t_0,\dots,t_n]/(\textstyle\sum_j t_j-1)
\]
be the algebraic $n$-simplex, with codimension-$1$ faces given by $\{t_i=0\}$. The \emph{faces} of $\Box^n$ or of $\Delta^n$ are the intersection of codimension-$1$ faces. We recall briefly the definition of the Bloch--Levine cycle complexes, in general position with respect to some closed subschemes:

\begin{definition}\label{def:complexes}
Let $s$ be a finite set of closed subschemes of $X$, each flat over $R$, with $X\in s$. Let $z^q_s(X,n)$ (resp. $z^q_s(X,n)^c$) be the subgroup of the group of codimension-$q$ cycles in $X\times\Delta^n$ (respectively $X\times\Box^n$), generated by integral closed subschemes $Z$ meeting $S\times_R F$ properly (i.e. with correct codimension) for every $S\in s$ and every face $F$. 
\end{definition}

Under the alternating sum of simplicial face maps (respectively cubical face maps, after taking the non-degenerate quotient as in \cite[\S3]{L1}) these form homological complexes $z^q_s(X,\bullet)$ and $z^q_s(X,\bullet)^c$. When $s=\{X\}$, we omit $s$ from the notation. Both complexes are contravariantly functorial for flat morphisms and covariantly functorial for proper morphisms. 

We recall now the main geometric ingredients of Levine's simplicial-cubical comparison over fields \cite[\S4]{L1}. These constructions are base-agnostic, and go through exactly as over fields. In particular, the main ``prism cycle'' construction and associated properties all work identically over $\Z$:

\begin{lemma}[{\cite[Lemma 4.1]{L1}}]\label{lem:Wn}
Let $W_n\subseteq\Box^{n+1}\times\PP^1$ be the subscheme defined by 
\[
T_0(1-t_n)(1-t_{n+1})=T_0-T_1,
\]
where $T_0,T_1$ are homogeneous coordinates on $\PP^1$, let $\pi_n: W_n\to\Box^n$ and $p_n:\Box^{n+1}\times\PP^1\to\Box^{n+1}$ be the maps of loc. cit., and for $Z\in z^q(X\times\Box^n)$ set $W^X_n(Z)=p_{n*}\big(\pi_n^{*}Z\big)$. Then $\pi_n$ is flat of relative dimension one, $W^X_n(-)$ preserves properness of intersections, and 
\[
W^X_n(Z)\cdot \{t_{n+1}=0\}=Z=W^X_n(Z)\cdot \{t_n=0\}.\]
\end{lemma}

As in loc. cit., we also have the cubical-simplicial complex $z^q_s(X,m,n)$ of cycles on $X\times\Box^m\times\Delta^n$ with the combination of conditions. Under the cubical differential $d'$ and the simplicial differential $d''$, this is a double complex whose total complex we will simply denote by $\mathrm{Tot}_X$ for short. It has a pair of augmentations
$$\epsilon':\mathrm{Tot}_X\to z^q_s(X,\bullet)^c,\;\;\; \epsilon'':\mathrm{Tot}_X\to z^q_s(X,\bullet)$$
and the spectral sequence degeneration calculation of \cite[Theorem 4.7]{L1} applies verbatim: 

\begin{proposition}\label{prop:reduction}
If the rows $(z^q_s(X,\bullet,n),d')$ and columns $(z^q_s(X,m,\bullet),d'')$ are acyclic for $m,n\ge1$, then $\epsilon',\epsilon''$ are quasi-isomorphisms. 
\end{proposition}

The acyclicity of the rows and columns is \cite[Lemma 4.6]{L1}, and only depends on the following two homotopy properties:
\begin{itemize}
\item (HC) The projection $p: X\times\A^1\to X$ induces a quasi-isomorphism 
\[
p^*: z^q_s(X,\bullet)^c\to z^q_{p^\bullet s}(X\times\A^1,\bullet)^c.\]
\item (HS) The map $p^*: z^q_s(X,\bullet)\to z^q_{p^\bullet s}(X\times\A^1,\bullet)$ is a quasi-isomorphism.
\end{itemize}

Under these assumptions, the proofs of \cite[Theorem 4.5]{L1} and \cite[Theorem 2.1]{B} apply verbatim, since the homotopies there are produced from $W^X_n$ (respectively the simplicial prism) and make sense even over $\Z$. The homotopy properties \emph{(HC)} and \emph{(HS)} we can deduce from the weak moving lemmas, which over fields are \cite[Proposition 4.4]{L1} and \cite[Lemma 2.3]{B}, the remaining steps being base-independent. We prove this as Theorem \ref{thm:wml} in the next subsection.

\section{Weak moving lemmas over semi--local PIDs}

In this setion, we imitate Levine's proof of the weak moving lemma:

\begin{theorem}[weak moving lemma over $R$]\label{thm:wml}
Let $X/R$ be smooth of pure relative dimension $d$, let $N\ge0$, and let $G=\mathbb G_{a,R}^N$ act on $\A^N_R$ by translation. Let $s$ be a finite set of closed subschemes of $X$, flat over $R$, with $X\in s$. Fix a set of subschemes $y=\{X\times H_1,\dots,X\times H_l\}$ with each $H_j\subseteq\A^N_R$ a nonempty closed subscheme flat over $R$. Then the inclusions
$$z^q_{y\cup p^* s}(X\times\A^N,\bullet)\ \hookrightarrow\ z^q_{p^* s}(X\times\A^N,\bullet),\;\;\;
z^q_{y\cup p^* s}(X\times\A^N,\bullet)^c\ \hookrightarrow\ z^q_{p^* s}(X\times\A^N,\bullet)^c$$
are quasi-isomorphisms.
\end{theorem}

\begin{remark}
The same statement for a general group $G$ acting on $X$ holds by the same method, but is not needed here.
\end{remark}

Set $M=X\times\A^N\times\Box^n$ (or $X\times\A^N\times\Delta^n$), smooth over $R$ of pure relative dimension $D=d+N+n$, so that $\dim M_K=\dim M_{k_i}=D$ and $\dim M=D+1$. We prove the cubical case; the simplicial case is identical.

First, note that any integral closed subscheme $W\subseteq M$ either is \emph{horizontal}, i.e. dominates $\mathrm{Spec}\,R$ and is thus flat over it, or maps to a closed point $\fp_i$ (is \emph{vertical} at $\fp_i$). Any cycle thus decomposes uniquely as $Z=Z^h+\sum_i Z^{v,i}$, and this decomposition is preserved by the translations $t_g$, $g\in G(R')$. A simple counting of dimensions tells us that a codimension-$q$ cycle $Z$ on $M$ meeting $C$ properly in $M$ reduces to checking separately that $Z^h$ meets $C_K$ properly in each of the fibers over $K$ and $k_i$, and that for each $i$, $Z^{v,i}$ meets $C_{k_i}$ properly in $M_{k_i}$.

To fix a translation over $R$ generic enough at all fibers, we will use the following immediate consequence of the Chinese remainder theorem: 

\begin{lemma}\label{lem:lifting}
Suppose all $k_i$ are infinite. Let $V\subseteq\A^N_K$ and, for each $i$, $V_i\subseteq\A^N_{k_i}$, be dense open subschemes. Then there is $\underline a\in R^N=\A^N(R)$ with generic fiber $\underline a_K\in V$ and reductions $\underline a_{k_i}=\underline a\bmod\fp_i\in V_i$ for all $i$.
\end{lemma}

\subsection{Proof of Theorem~\ref{thm:wml} when all $k_i$ are infinite}

We will need the following transcendental extension of $R$ to make our genericity argument work: let $S=R[t_1,\dots,t_N]$ and let $R'=R(t_1,\dots,t_N)$ be the localization of $S$ at the multiplicative set $S\setminus\bigcup_i\fp_iS$. Then $R'$ is a semi-local PID, flat over $R$, with maximal ideals $\fp_iR'$, residue fields $k_i(t_1,\dots,t_N)$ and fraction field $K(t_1,\dots,t_N)$; the parameters $t_j$ are transcendental over $K$ and their residues are transcendental over each $k_i$. Write $\pi:\mathrm{Spec}\,R'\to\mathrm{Spec}\,R$ for the projection.

We may now establish the theorem in the infinite residue field case: set 
\[
C=z^q_{p^* s}(X\times\A^N,\bullet)^c\big/\,z^q_{y\cup p^* s}(X\times\A^N,\bullet)^c\]
and let $C'$ be the analogous construction for the base change to $R'$. We will show that $C$ is acyclic by that the pullback $\pi^*: C\to C':=C\otimes_R R'$ is both null-homotopic and an injection on homology.

Let $\psi:\A^1_{R'}\to G_{R'}=\mathbb G_{a,R'}^N$, $\psi(\lambda)=(1-\lambda)\underline t$, so $\psi(1)=0$ is the identity translation and $\psi(0)$ is translation by the generic vector $\underline t$; set $\phi(z,\lambda)=(\psi(\lambda)\cdot z,\lambda)$. The operators $h_n$ of \cite[(4.2)]{L1}, which make equal sense over $R$, define a chain homotopy
$$d\,h_n+h_{n-1}\,d=\pi^*-\psi(0)_*\circ\pi^*,$$
provided each term of $h_n(Z)$ - namely $Z\times\A^1$, $\phi(Z\times\A^1)$, $W^X_n(dZ\times\A^1)$ and $W^X_n(\phi(dZ\times\A^1))$ - meets the faces properly over $R'$. As noted previously, this can be checked separately over each of the fibers, which are copies of the relevant ambient space over the infinite fields $K(\underline t)$ and $k_i(\underline t)$, over each of which $\psi(0)=\underline t$ is the generic translation. Thus, this follows by the same $k$-genericity argument of \cite[Lemma 4.2]{L1}. The same argument shows that $\psi(0)_*\pi^* Z$ - a generic translate, fiberwise - additionally meets every $X\times H_j\times F$ (for $F$ a face) properly, so it lies in $z^q_{y\cup p^* s}(X_{R'}\times\A^N)^c$, i.e. $\psi(0)_*\circ\pi^*=0$ in $C'$. Hence $\pi^*$ is null-homotopic and is zero on homology.

It remains to show $ \pi^*$ is injective on homology: let $[Z]\in H_n(C)$ satisfy $\pi^*[Z]=0$, so $\pi^* Z=dW$ in $C'$ for some $W$. Each cycle over $R'$ is defined over a localization $R[\underline t]_f$ with $f\notin\bigcup_i\fp_i R[\underline t]$; spreading $W$ out we obtain a cycle $\mathcal W$ over a dense open $\mathcal U\subseteq\A^N_R$ - namely, the locus of parameters $\underline a$ at which $\mathcal W$ and all the finitely many face-intersections occurring in $dW$ are in proper position relative to the section $X\times\{\underline a\}\times\Box^{n+1}$. This is an open dense condition in each fiber, so by Lemma \ref{lem:lifting} we may find a global point $\underline a\in\mathcal U(R)$; the restriction $\mathrm{sp}_{\underline a}$ to the section visibly commutes with the face maps. Since $\pi^* Z=Z\times_R\A^N$ is the constant family, $\mathrm{sp}_{\underline a}\pi^* Z=Z$, and therefore in $C$
$$Z=\mathrm{sp}_{\underline a}(dW)=d\,(\mathrm{sp}_{\underline a}W),$$
and $[Z]=0$. The result follows.

\subsection{Descending to the finite residue field case} \label{section:descent}

When the residue fields $k_i$ are finite the preceding lifting lemma can fail, and so we have to descend from auxiliary $\ell$-adic towers of extensions, just as Levine does over fields \cite[Lemma 4.3]{L1} to reduce to the infinite case, then use the push-pull formula for transfers.

\begin{lemma}
Fix a prime $\ell$. There exists a nested tower of domains $R\subset R_1 \subset R_2\subset \ldots$ such that each successive cover is \'{e}tale of degree $\ell$ and totally inert at each prime $\fp_i$.
\end{lemma}
\begin{proof}
    Noting that an \'{e}tale domain covering $R$ is also a semi-local PID, it suffices by induction to construct the extension $R_1/R$ with the requisite properties. Indeed, we may simply take a monic irreducible polynomial $f_i$ of degree $\ell$ over $k_i$, for each $i$ such that $k_i$ is finite, and use the Chinese remainder theorem to find $f\in R[t]$ whose reductions are each of the $f_i$; then $R_1=R[t]/(f)$ satisfies our requirements. 
\end{proof}

\begin{lemma}\label{lem:pid}
$R_\infty:=\varinjlim_n R_n$ is a semi-local principal ideal domain with maximal ideals $\mathfrak P_1,\dots,\mathfrak P_r$ over $\fp_1,\dots,\fp_r$ with all residue fields infinite.
\end{lemma}
\begin{proof}
$R_\infty$ is a one-dimensional normal domain. Each localization $(R_\infty)_{\mathfrak P_i}=\varinjlim_n(R_n)_{\fp_i}$ is also a DVR since the tower is unramified; its residue field is the infinite field $\bigcup_n k_i^{(n)}$. Thus $R_\infty$ is a one-dimensional semi-local Pr\"ufer domain whose localizations are discrete. Given $0\neq I\subseteq R_\infty$, set $n_i=\min\{v_i(x):x\in I\}$ for the valuation $v_i$ of $\mathfrak P_i$, choose $a_i\in I$ with $v_i(a_i)=n_i$, and put $J=(a_1,\dots,a_r)\subseteq I$. As a finitely generated ideal of a Pr\"ufer domain $J$ is invertible, and an invertible ideal of a semi-local domain is principal, $J=(x)$ with $v_i(x)=n_i$. For any $y\in I$ we have $v_i(y/x)\ge0$ for all $i$, so $y/x\in\bigcap_i(R_\infty)_{\mathfrak P_i}=R_\infty$, whence $y\in(x)$ and $I=J=(x)$. Thus every ideal is principal, so $R_\infty$ is a PID.
\end{proof}

By Lemma \ref{lem:pid} and the infinite-residue-field case, the quotient complex $C_{R_\infty}$ is acyclic. The cycle complexes commute with filtered colimits, so $C_{R_\infty}=\varinjlim_n C_{R_n}$ and $H_\bullet C_{R_\infty}=\varinjlim_n H_\bullet C_{R_n}=0$. Let $\alpha\in H_\bullet C_R$; its image in $H_\bullet C_{R_\infty}$ vanishes, so $\pi_n^*\alpha=0$ in $H_\bullet C_{R_n}$ for some $n$, where $\pi_n:\operatorname{Spec}R_n\to\operatorname{Spec}R$. The transfer $\pi_{n*}:C_{R_n}\to C_R$ satisfies $\pi_{n*}\pi_n^*=[R_n:R]\cdot\mathrm{id}=\ell^n\cdot\mathrm{id}$; hence $\ell^n\alpha=\pi_{n*}\pi_n^*\alpha=0$. Doing this for two distinct primes $\ell\neq\ell'$ annihilates $\alpha$ by the coprime integers $\ell^n$ and $\ell'^{\,n'}$, so $\alpha=0$. Hence $H_\bullet C_R=0$. 

As noted, the simplicial part of Theorem~\ref{thm:wml} follows identically.

\section{Comparison statements}

\begin{corollary}\label{cor:homotopy}
The homotopy properties \emph{(HC)} and \emph{(HS)} of Proposition~\ref{prop:reduction} hold over $R$.
\end{corollary}

\begin{proof}
Apply Theorem~\ref{thm:wml} with $\A^N=\A^1$ and $H_1=\{0\}$, $H_2=\{1\}$. The proofs of \cite[Theorem 4.5]{L1} in the cubical case and \cite[Theorem 2.1]{B} in the simplicial case then go through without change.
\end{proof}

We immediately deduce our main theorem:

\begin{theorem}\label{thm:main}
Let $R$ be a semi-local principal ideal domain where all residue fields are either finite or infinite, and $X$ a smooth $R$-scheme of pure relative dimension $d$. Then for every $q$ we have a zig-zag of quasi-isomorphisms
$$z^q(X,\bullet)^c \xleftarrow{\epsilon'} \mathrm{Tot}_X \xrightarrow{\epsilon''} z^q(X,\bullet)$$
and in particular $\mathrm{CH}^q(X,n)^c\cong\mathrm{CH}^q(X,n)$ for all $n,q$. More generally, the same holds for $z^q_s(X,\bullet)$ for any finite family $s$ of closed subschemes flat over $R$ with $X\in s$.
\end{theorem}
\begin{proof}
By Proposition \ref{prop:reduction}, the rows and columns of the double complex are acyclic in positive degree.
\end{proof}

Let $B$ be a Dedekind domain. The various Bloch cycle groups sheafify over $B$ by Zariski descent of algebraic cycles. Then the zig-zag of maps
$$z^q(X,\bullet)^c \xleftarrow{\epsilon'} \mathrm{Tot}_X \xrightarrow{\epsilon''} z^q(X,\bullet)$$
sheafifies into a zig-zag of maps between the Zariski sheaves on $B$
 \[
    U\mapsto z^q(X_U,\bullet)^c,\;\;\;  U\mapsto z^q(X_U,\bullet),
 \]
which is a quasi-isomorphism stalk-wise (which are DVRs). These thus compute the same hypercohomology groups; the latter were described by Levine in \cite{L2}, and many properties established. These compute the same motivic cohomology groups as modern Voevodsky-style definitions of motivic cohomology for smooth schemes over Dedekind bases. If $B$ is semi-local, then the hypercohomology of the cubical Bloch sheaf is computed by the global sections $z^q(X_B,\bullet)^c$, by Theorem \ref{thm:main}.

{
\printbibliography

@article{B,
  author={Bloch, Spencer}, title={Algebraic cycles and higher {$K$}-theory},
  journal={Adv. Math.}, volume={61}, number={3}, year={1986}, pages={267--304}, shorthand="Bl"}

@article{L1,
  author={Levine, Marc}, title={Bloch's higher {C}how groups revisited},
  journal={Ast\'erisque}, volume={226}, year={1994}, pages={235--320},
  note={$K$-theory (Strasbourg, 1992)},shorthand="L1"}

@article{L2,
  author  = {Marc Levine},
  title   = {Techniques of localization in the theory of algebraic cycles},
  journal = {Journal of Algebraic Geometry},
  volume   = {10},
  number   = {2},
  year     = {2001},
  pages    = {299--363},
  shorthand = "L2"
}

@unpublished{X,
  author = {Peter Xu},
  title = {Symbols for toric Eisenstein cocycles and arithmetic applications},
  note = "Submitted.",
  shorthand = "X"
}

@article{G,
  author={Geisser, Thomas}, title={Motivic cohomology over {D}edekind rings},
  journal={Math. Z.}, volume={248}, number={4}, year={2004}, pages={773--794}, shorthand="G"}
}

\end{document}